\documentclass[runningheads]{llncs}
\usepackage{graphicx}

\usepackage{amsmath} 
\usepackage{amssymb}  
\usepackage{xcolor}
\usepackage{caption}
\usepackage{graphicx}
\usepackage{mathtools}
\usepackage{caption}
\usepackage{subcaption}
\usepackage{floatrow}
\usepackage{tikz}
\usepackage{url}

\usepackage{algorithmicx}
    \usepackage[ruled]{algorithm}
    \usepackage{algpseudocode}

\usepackage{pgfplots}
\usepackage{xcolor}
\usetikzlibrary{matrix,arrows,calc,positioning,shapes,decorations.pathreplacing}
\usepackage{graphicx}

\usepackage{enumerate}
\usepackage[all,tips]{xy}
\SelectTips{cm}{11}

\newcommand{\R}{\mathbb{R}}

\DeclareMathOperator{\spn}{span}

\begin{document}
\title{Virtual nonlinear
nonholonomic constraints from a symplectic point of view.\thanks{The authors acknowledge financial support from Grant PID2019-106715GB-C21 funded by MCIN/AEI/ 10.13039/501100011033. A.B. was partially supported by NSF grants DMS-1613819 and DMS-2103026, and AFOSR grant FA
9550-22-1-0215.}}

\author{Efstratios Stratoglou\inst{1} \and
Alexandre Anahory Simoes\inst{2} \and
Anthony Bloch\inst{3}\and
Leonardo Colombo\inst{4}}
\authorrunning{E. Stratoglou et al.}
%
\institute{Universidad Polit\'ecnica de Madrid (UPM), José Gutiérrez Abascal, 2, 28006 Madrid, Spain. \email{ef.stratoglou@alumnos.upm.es } \and
School of Science and Technology, IE University, Spain.
\email{alexandre.anahory@ie.edu}\\ \and
Department of Mathematics, University of Michigan, Ann Arbor, MI 48109, USA. 
\email{abloch@umich.edu}\and Centre for Automation and Robotics (CSIC-UPM), Ctra. M300 Campo Real, Km 0,200, Arganda
del Rey - 28500 Madrid, Spain. \email{leonardo.colombo@csic.es}}
\maketitle              
\begin{abstract}
   In this paper, we provide a geometric characterization of virtual nonlinear nonholonomic constraints from a symplectic perspective.  Under a transversality assumption, there is a unique control law making the trajectories of the associated closed-loop system satisfy the virtual nonlinear nonholonomic constraints. We characterize them in terms of the symplectic structure on $TQ$ induced by a Lagrangian function and the almost-tangent structure. In particular, we show that the closed-loop vector field satisfies a geometric  equation of Chetaev type. Moreover, the closed-loop dynamics is obtained as the projection of the uncontrolled dynamics to the tangent bundle of the constraint submanifold defined by the virtual constraints.
 
\keywords{Nonholonomic Systems  \and Virtual Constraints \and Geometric Control \and Nonlinear Control.}
\end{abstract}
\section{Introduction}



Virtual constraints are relationships imposed on control systems through feedback control instead of physical links between joints. Virtual nonholonomic constraints \cite{griffin2015nonholonomic} represent a specific class of virtual constraints that depend on the system’s velocities in addition to its configurations.

In the paper \cite{virtual}, we developed a geometric description of linear virtual nonholonomic constraints, i.e., constraints that are linear in the velocities, while in \cite{affine} we addressed the problem of affine virtual nonholomonic constraints. In \cite{nonlinear}, we extended the latest outcomes to the setting of virtual nonlinear nonholonomic constraints. By imposing a transversality condition, we identified a unique control law ensuring that the closed-loop system’s trajectories comply with the prescribed constraints.

In this article, we explore a geometric perspective on virtual nonlinear nonholonomic constraints within a symplectic framework, laying the geometric foundations for studying their properties. The closed-loop system is expressed in terms of the symplectic structure on $TQ$, induced by a Lagrangian function, and the almost-tangent structure. In particular, we demonstrate that the closed-loop vector field satisfies a Chetaev-type geometric equation. Additionally, we show that the resulting closed-loop dynamics can be interpreted as the projection of the uncontrolled system onto the tangent bundle of the submanifold defined by the virtual constraints. Lastly, we give an explicit application of the theory to the example of a double pendulum and we test our results with numerical simulations.


\section{Preliminaries on differential geometry}

Suppose $Q$ is a differentiable manifold of dimension $n$. Throughout the text, $q^{i}$ will denote a particular choice of local coordinates on this manifold and $TQ$ denotes its tangent bundle, with $T_{q}Q$ denoting the tangent space at a specific point $q\in Q$ generated by the coordinate vectors $\frac{\partial}{\partial q^{i}}$. $v_{q}$ denotes a vector at $T_{q}Q$ and, $(q^{i},\dot{q}^{i})$ denote natural coordinates on $TQ$. The map $\tau_{Q}:TQ \rightarrow Q$, is the canonical projection sending each vector $v_{q}$ to the corresponding base point $q$. In coordinates $\tau_{Q}(q^{i},\dot{q}^{i})=q^{i}$. The tangent map of the canonical projection is given by $T\tau_Q:TTQ\to TQ.$ The cotangent bundle of $Q$ is denoted by $T^*Q$ and for $q\in Q$ the cotangent space $T^*_qQ$ is generated by cotangent vectors $dq^i$ which satisfies $\langle dq^i,\frac{\partial}{\partial q^j}\rangle=\delta_{ij}$, where $\delta_{ij}$ is the Kronecker delta.

A vector field $X$ on $Q$ is a map assigning to each point $q\in Q$ a vector tangent to $q$, that is, $X(q)\in T_{q}Q$. In the context of mechanical systems, we find a special type of vector fields that are always defined on the tangent bundle $TQ$, considered as a manifold itself. A second-order vector field (SODE) $\Gamma$ on the tangent bundle $TQ$ is a vector field on the tangent bundle satisfying the property that $T\tau_{Q}\left(\Gamma (v_{q})\right) = v_{q}$. The expression of any SODE in local coordinates is 
$\Gamma(q^{i},\dot{q}^{i})= \dot{q}^{i}\frac{\partial}{\partial q^{i}} + f^{i}(q^{i},\dot{q}^{i}) \frac{\partial}{\partial \dot{q}^{i}}$, 
where $f^{i}:TQ \rightarrow \mathbb{R}$ are $n$ smooth functions. We denote the set of all vector fields on $Q$ by $\mathfrak{X}(Q)$.

A one-form $\alpha$ on $Q$ is a map assigning to each point $q$ a cotangent vector to $q$, that is, $\alpha(q)\in T^{*}Q$. Cotangent vectors acts linearly on vector fields according to $\alpha(X) = \alpha_{i}X^{i}\in \mathbb{R}$ if $\alpha = \alpha_{i}dq^{i}$ and $X = X^{i} \frac{\partial}{\partial q^{i}}$. {We denote the set of all one-forms on $Q$ by $\Omega^{1}(Q)$. In the following, we will denote two-forms or $(0,2)$-tensor fields as those skew-symmetrix bilinear maps that act on a pair of vector fields and produce a scalar number; and also we define  $(1,1)$-tensor fields to be  linear maps that act on a vector field and produce a new vector field.}

In a symplectic manifold $(Q,\omega)$ we have a linear isomorphism $\flat_{\omega}:\mathfrak{X}(Q)\rightarrow \Omega^{1}(Q)$, given by $\langle\flat_{\omega}(X),Y\rangle=\omega(X,Y)$ for any vector fields $X, Y$. The inverse of $\flat_{\omega}$ will be denoted by $\sharp_{\omega}$.

In the paper we will use the notion of canonical almost tangent structure $J:TTQ \rightarrow TTQ$. This is a $(1,1)$- tensor field on $TQ$ whose expression in local coordinates is
$J=dq^{i}\otimes \frac{\partial}{\partial \dot{q}^{i}}$, where $\otimes$ stands for the tensor product. For instance, if $\Gamma$ is a SODE  $J(\Gamma) = \dot{q}^{i}\frac{\partial}{\partial \dot{q}^{i}}$. Considering the dual of the canonical almost tangent structure and a function $\phi\in C^{\infty}(TQ)$, we have that $J^{*}(d\phi) = \frac{\partial \phi}{\partial \dot{q}^{i}}dq^{i}$.

Given a Lagrangian function $L:TQ\rightarrow \mathbb{R}$, the associated energy $E_{L}$ is the function defined by $E_{L}(q,\dot{q})=\dot{q}\frac{\partial L}{\partial \dot{q}} - L(q,\dot{q})$ and we may write a symplectic form on $TQ$, denoted by $\omega_{L}$, defined by $\omega_L = -d(J^{*}(dL))$. In natural coordinates of $TQ$, $\omega_{L}=\frac{\partial^{2} L}{\partial \dot{q}^{i} \partial q^{j}} dq^{i}\wedge dq^{j} + \frac{\partial^{2} L}{\partial \dot{q}^{i} \partial \dot{q}^{j}} dq^{i}\wedge d\dot{q}^{j}$. This geometric construction is used to write Euler-Lagrange dynamics as the integral curves of the vector field $\Gamma_{L}$ solving the equation $i_{\Gamma_{L}}\omega_{L}=dE_{L}$, where $i_{\Gamma_{L}}\omega_{L}$ denotes the contraction of $\Gamma_L$ by $\omega_L$ (see \cite{Leon_Rodrigues}). In fact, this is the geometric equation defining Hamiltonian vector fields on general symplectic manifolds.

Let $\mathcal{G}$ be a Riemannian metric on $Q$, locally given by
$\mathcal{G}_{i j}$ $=\mathcal{G}\left(\frac{\partial}{\partial q^{i}},\frac{\partial}{\partial q^{j}}\right)$.

Let $\nabla$ be a linear connection. In local coordinates, connections are fully described by the Christoffel symbols which are real-valued functions on $Q$ given by $\nabla_{\frac{\partial}{\partial q^{i}}}\frac{\partial}{\partial q^{j}}=\Gamma_{i j}^{k}\frac{\partial}{\partial q^{k}}.$
Thus if $X$ and $Y$ are vector fields whose coordinate expressions are $X=X^{i}\frac{\partial}{\partial q^{i}}$ and $Y=Y^{i}\frac{\partial}{\partial q^{i}}$, then $\nabla_{X} Y=\left(X^{i}\frac{\partial Y^{k}}{\partial q^{i}}+X^{i}Y^{j} \Gamma_{i j}^{k}\right)\frac{\partial}{\partial q^{k}}.$

From now on if $(Q,\mathcal{G})$ is a Riemannian manifold, $\nabla$ will be the \textit{Levi-Civita connection} and the covariant derivative of a vector field $X\in \mathfrak{X}(Q)$ along a curve $q:I\rightarrow Q$, where $I$ is an interval of $\mathbb{R}$, is given by the local expression
\begin{equation*}
		\nabla_{\dot{q}}X (t)=\left( \dot{X}^{k}(t)+\dot{q}^{i}(t) X^{j}(t)\Gamma_{i j}^{k}(q(t)) \right)\frac{\partial}{\partial q^{k}}.
\end{equation*}

Finally, we can use the non-degeneracy property of $\mathcal{G}$ to define the musical isomorphism $\flat_{\mathcal{G}}:\mathfrak{X}(Q)\rightarrow \Omega^{1}(Q)$ defined by $\flat_{\mathcal{G}}(X)(Y)=\mathcal{G}(X,Y)$ for any $X, Y \in \mathfrak{X}(Q)$. Also, denote by $\sharp_{\mathcal{G}}:\Omega^{1}(Q)\rightarrow \mathfrak{X}(Q)$ the inverse musical isomorphism, i.e., $\sharp_{\mathcal{G}}=\flat_{\mathcal{G}}^{-1}$. In local coordinates, $\flat_{\mathcal{G}}(X^{i}\frac{\partial}{\partial q^{i}})=\mathcal{G}_{ij}X^{i}dq^{j}$ and $\sharp_{\mathcal{G}}(\alpha_{i}dq^{i})=\mathcal{G}^{ij}\alpha_{i}\frac{\partial}{\partial q^{j}}$, where $\mathcal{G}^{ij}$ is the inverse matrix of $\mathcal{G}_{ij}$. The gradient of a smooth function $f$ is defined as the vector field $\text{grad }f=\sharp_{\mathcal{G}}(df)$.

The complete lift of $\mathcal{G}$ on $Q$ is denoted by $\mathcal{G}^{c}$ and it is a semi-Riemannian metric, since it is not positive-definite. In natural bundle coordinates, its expression is $\displaystyle{\mathcal{G}^{c}=\dot{q}^{k}\frac{\partial \mathcal{G}_{ij}}{\partial q^{k}} dq^{i} \otimes dq^{j} + \mathcal{G}_{i j} dq^{i} \otimes d\dot{q}^{j} + \mathcal{G}_{i j} d\dot{q}^{i} \otimes dq^{j}}$.



{\begin{definition}
\begin{enumerate}
    \item The vertical lift of a vector field $X\in \mathfrak{X}(Q)$ to $TQ$ is a vector field on the tangent bundle $TQ$. If $(q^{i},\dot{q}^{i})$ are the natural coordinates on $TQ$ and $X=X^{i} \frac{\partial}{\partial q^{i}}$ then its local expression is $X^{V}=X^{i}\frac{\partial}{\partial \dot{q}^{i}}$.
    \item The  complete lift of a vector field, $X$, which in local coordinates is given by $X=X^i\frac{\partial}{\partial q^i}$ is $X^c=X^i\frac{\partial}{\partial q^i} + \dot{q}^j\frac{\partial X^i}{\partial q^j}\frac{\partial}{\partial \dot{q}^i}$. 
    \item The vertical lift of a one-form $\alpha\in\Omega^1(Q)$ is defined as the pullback of $\alpha$ to $TQ$, i.e. $\alpha^V=(\tau_Q)^*\alpha,$ which locally is $\alpha^V=\alpha_idq^i$ .
    \item The complete lift of a one-form $\alpha\in\Omega^1(Q)$ is 
    $\alpha^c=\dot{q}^j\frac{\partial\alpha^i}{\partial q^j}dq^i + \alpha^idq^i$.
\end{enumerate}
\end{definition}

\begin{lemma}\label{completemetricLemma}
    For a Riemannian metric $\mathcal{G}$ on $Q$, vector  fields  $X,Y\in\mathfrak{X}(Q)$ and a one-form $\alpha\in\Omega^1(Q)$ we have
    \begin{enumerate}
        \item $(\alpha(X))^V=\alpha^c(X^V),\, \mathcal{G}^c(X^V,Y^c) = \mathcal{G}^c(X^c,Y^V) = [\mathcal{G}(X,Y)]^V,\newline \mathcal{G}^c(X^V,Y^V)=0.$
        \item $\left[ \sharp_{\mathcal{G}}(\alpha)\right]^{V} = \sharp_{\mathcal{G}^{c}}(\alpha^{V})$.
    \end{enumerate}
\end{lemma}
\begin{proof}
   For $1$, see \cite{Leon_Rodrigues}. For $2$, given any $Y\in\mathfrak{X}(Q)$, it is enough to prove the equality using the inner product with the lifts $Y^{c}$ and $Y^{V}$, because if $\{Y^{a}\}$ was a local basis of vector fields, then $\{(Y^{a})^{c}, (Y^{a})^{V}\}$ would also be a local basis of vector fields on $TQ$. On one hand, $\displaystyle{\mathcal{G}^{c}\left(\left[ \sharp_{\mathcal{G}}(\alpha)\right]^{V},Y^{V} \right) = 0 = \alpha^{V}(Y^{V}) = \mathcal{G}^{c}\left(\sharp_{\mathcal{G}^{c}}(\alpha^{V}),Y^{V} \right)}$. On the other hand, $\mathcal{G}^{c} \left(\left[ \sharp_{\mathcal{G}}(\alpha)\right]^{V}, Y^{c} \right) = \left[ \mathcal{G}(\sharp_{\mathcal{G}}(\alpha), Y)\right]^{V} = \left[ \alpha(Y)\right]^{V}= \alpha^{V}(Y^{c})=$\newline $\mathcal{G}^{c} \left( \sharp_{\mathcal{G}^{c}}(\alpha^{V}),Y^{c} \right)$. Hence, the results follow by non-degeneracy of $\mathcal{G}^{c}$.\hfill$\square$\end{proof}

\section{Nonlinear nonholonomic mechanics}

A nonlinear nonholonomic constraint on a mechanical system is represented by a submanifold $\mathcal{M}$ of the tangent bundle $TQ$. The constraint may be written as the set of points where a function of the type $\phi:TQ \rightarrow \mathbb{R}^{m}$ vanishes, where $m < n=\dim Q$. That is, $\mathcal{M}=\phi^{-1}(\{0\})$. If every point in $\mathcal{M}$ is regular, i.e., the tangent map $T_{p}\phi$ is surjective for every $p\in \mathcal{M}$, then $\mathcal{M}$ is a submanifold of $TQ$ with dimension $2n-m$ by the regular level set theorem. 

Let $\phi = (\phi^{1}, \dots, \phi^{m})$ denote the coordinate functions of the constraint $\phi$. The equations of motion of a nonlinear nonholonomic system are integral curves of a vector field $\Gamma_{nh}$ defined by the equations
\begin{equation}\label{noneq}
i_{\Gamma_{nh}}\omega_{L} - dE_{L} = \lambda_{a}J^{*}(d\phi^{a}),\quad\Gamma_{nh} \in TM,
\end{equation}
where $\lambda_{a}$ are Lagrange multiplier's to be determined. These equations have a well-defined solution if $\sharp_{\omega_{L}}(J^{*}(d\phi^{a}))\cap TM = \{0\}$.

The coordinate expression of the equations of motion of a system with nonlinear nonholonomic constraints are called Chetaev's equations and they are given by  (see \cite{cendra}, \cite{MdLeon} for more details)
\begin{equation}
        \frac{d}{dt}\left(\frac{\partial L}{\partial \dot{q}}\right)-\frac{\partial L}{\partial q}=\lambda_{a} \frac{\partial \phi^{a}}{\partial \dot{q}},\quad\phi^{a}(q,\dot{q}) = 0.
\end{equation}

In the following, we will consider a mapping that to each point $v_{q}$ on the submanifold $\mathcal{M}$ assigns a vector subspace of $T_{v_{q}}(TQ)$. From now on, let $S$ be a distribution on $TQ$ restricted to $\mathcal{M}$, whose annihilator is spanned by the one-forms $J^{*}(d\phi^{a})$, i.e., $S^{o}=\left\langle \{J^{*} (d\phi^{a})\}\right\rangle.$ Then, Chetaev's equations may be written in Riemannian form using a geodesic-type equation according to the following theorem

\begin{proposition}
    A curve $q:I\rightarrow Q$ is a solution of Chetaev's equations for a mechanical type Lagrangian, i.e., of the type $L(q,\dot{q})=\frac{1}{2}\mathcal{G}_{ij}\dot{q}^{i}\dot{q}^{j} - V(q)$, if and only if $\phi(q,\dot{q})=0$ and it satisfies the equation
    \begin{equation}\label{Chetaev's eqns}
        \left( \nabla_{\dot{q}}\dot{q} + \text{grad } V \right)^{V} \in S^{\bot},
    \end{equation}
    where $S^{\bot}$, the orthogonal distribution to $S$, with respect to the semi-Riemannian metric $\mathcal{G}^{c}$.
\end{proposition}

\begin{proof}

    Chetaev's equations for mechanical type Lagrangians imply the following second-order differential equation
    $$\ddot{q}^{i} - \mathcal{G}^{ij}\left[ \frac{1}{2}\frac{\partial \mathcal{G}_{lk}}{\partial q^{j}}\dot{q}^{l}\dot{q}^{k}-\frac{\partial \mathcal{G}_{lj}}{\partial q^{k}}\dot{q}^{l}\dot{q}^{k} - \frac{\partial V}{\partial q^{j}}\right] = \lambda_{a}\mathcal{G}^{ij}\frac{\partial \phi^{a}}{\partial \dot{q}^{j}},$$
    where $\mathcal{G}^{ij}$ is the inverse matrix of $\mathcal{G}_{ij}$. The left-hand side can be recognized to be the coordinate expression of the vertical lift of the vector field $\nabla_{\dot{q}}\dot{q} + \text{grad} V$
    (see \cite{B&L} for details). We will show that the right-hand side is the coordinate expression of the vector field $\sharp_{\mathcal{G}^{c}}(J^{*}(d\phi^{a}))$. 
    
    Given a one-form $\alpha$ on $TQ$, the inverse musical isomorphism $\sharp_{\mathcal{G}^{c}}(\alpha)$ is characterized by
    $\mathcal{G}^{c}(\sharp_{\mathcal{G}^{c}}(\alpha), X) = \langle \alpha, X \rangle, \quad \text{for any } X\in \mathfrak{X}(TQ).$ Using this property, and taking  into account the coordinate expression of $J^{*}(d\phi^{a})$, we can deduce from $        \langle J^{*}(d\phi^{a}), \frac{\partial}{\partial q^{j}}\rangle = \frac{\partial \phi^{a}}{\partial \dot{q}^{j}}$ and $           \langle J^{*}(d\phi^{a}), \frac{\partial}{\partial \dot{q}^{j}}\rangle = 0$
    that $\sharp_{\mathcal{G}^{c}} (J^{*}(d\phi^{a})) = \mathcal{G}^{ij}\frac{\partial \phi^{a}}{\partial \dot{q}^{j}}\frac{\partial}{\partial \dot{q}^{j}} \in \sharp_{\mathcal{G}^{c}}(S^{o})$.  In addition, we have that $S^{\bot}$ satisfies $S^{\bot} = \sharp_{\mathcal{G}^{c}}(S^{o})$, which finishes the proof.\hfill$\square$
\end{proof}

\begin{remark}
 Notice that $S^{\bot}$ is spanned by vertical vectors in $TQ$. This observation will be relevant later in the paper.
\end{remark}

\section{Virtual nonlinear nonholonomic constraints}\label{sec:controler}
Next, we present the construction of virtual nonlinear nonholonomic  constraints. In contrast to the case of standard constraints on mechanical systems, the concept of virtual constraint is always associated with a controlled system and not just with a submanifold defined by the constraints.

Given a control force $F:TQ\times U \rightarrow T^{*}Q$ of the form $\displaystyle{F(q,\dot{q},u) = \sum_{a=1}^{m} u_{a}f^{a}(q)}$, where $f^{a}\in \Omega^{1}(Q)$ with $m<n$, $U\subset\mathbb{R}^{m}$ the set of controls and $u_a\in\mathbb{R}$ with $1\leq a\leq m$ the control inputs, consider the associated mechanical control system 
\begin{equation}\label{mechanical:control:system}
    \nabla_{\dot{q}}\dot{q} =-\text{grad } V+u_{a}Y^{a}(q),
\end{equation}
where $V$ is a potential function on $Q$, $Y^{a}=\sharp_{\mathcal{G}} (f^{a}(q))$ and $\mathcal{G}$ is a Riemannian metric. The distribution $\mathcal{F}\subseteq TQ$ generated by the vector fields  $Y^{a}=\sharp_{\mathcal{G}}(f^{a})$ is called the \textit{input distribution} associated with the mechanical control system \eqref{mechanical:control:system}.

Hence, the solutions of the previous equation are the trajectories of a vector field of the form
\begin{equation}\label{SODE}\Gamma(q, \dot{q}, u)=G(q,\dot{q})+u_{a}(Y^{a})_{(q,\dot{q})}^{V}.\end{equation}
We call each $Y^{a}=\sharp_{\mathcal{G}}(f^{a})$ a control force vector field, $G$ is the vector field determined by the unactuated forced mechanical system
$\nabla_{\dot{q}}\dot{q} =-\text{grad }V$.

Throughout the paper, we will restrict to the case where $G$ is the Euler-Lagrange dynamics determined by a mechanical type Lagrangian $L$ associated with $\mathcal{G}$ and a potential function $V$.

Now, we recall the concept of virtual nonholonomic constraint.

\begin{definition}A \textit{virtual nonholonomic constraint} associated with the mechanical control system \eqref{mechanical:control:system} is a controlled invariant submanifold $\mathcal{M}\subseteq TQ$ for that system, that is, 
there exists a control function $\hat{u}:\mathcal{M}\rightarrow \mathbb{R}^{m}$ such that the solution of the closed-loop system satisfies $\psi_{t}(\mathcal{M})\subseteq \mathcal{M}$, where $\psi_{t}:TQ\rightarrow TQ$ denotes its flow.\end{definition}

\begin{definition}
    Two subspaces $W_1$ and $W_2$ of a vector space $V$ are transversal if 
    \begin{enumerate}
        \item $V = W_1+W_2$
        \item $\dim V =\dim W_1 + \dim W_2$, i.e. the dimensions of $W_1$ and $W_2$ are complementary with respect to the ambient space dimension.
    \end{enumerate}
\end{definition}

Now, we state the theorem guaranteeing the existence and uniqueness of a control law forcing the system to comply with the virtual constraints, the proof can be found at \cite{nonlinear}.

\begin{theorem}\label{main:theorem}
If the tangent space, $T_{v_{q}}\mathcal{M}$, of the manifold $\mathcal{M}$ and the vertical lift of the control input distribution $\mathcal{F}$ are transversal and $T_{v_{q}}\mathcal{M}\cap \mathcal{F}^V=\{0\}$, then there exists a unique control function making $\mathcal{M}$ a virtual nonholonomic constraint associated with the mechanical control system \eqref{mechanical:control:system}.
\end{theorem}

From now on, we assume that $T\mathcal{M}$ and $\mathcal{F}^V$ are transversal. In the following result, we characterize the closed-loop dynamics arising from the previous theorem as trajectories of a vector field satisfying a similar equation to that of the nonholonomic equations \eqref{noneq}.

\begin{lemma}\label{lemma:similar:sharps}
    Consider the Lagrangian function $L:TQ\to\R$ of mechanical type $L=K(q,\dot{q})-V(q)$, where the kinetic energy is given by a Riemannian metric $\mathcal{G}$ and consider the symplectic form $\omega_L$ of $TQ$. For any one-form $f\in\Omega^1(Q)$ we have that $\sharp_{\omega_L}(f^V)=\sharp_{\mathcal{G}^c}(f^V)$.
\end{lemma}
\begin{proof}
    In local coordinates the flat map of  $\omega_L$ is represented by the matrix $\flat_{\omega_L}=\begin{pmatrix}
        A & \mathcal{G}_{ij}\\
        -\mathcal{G}_{ij} & 0
    \end{pmatrix},$ where $A$ is a skew symmetric matrix and $\mathcal{G}_{ij}$ is the matrix representation of the Riemannian metric. Hence, the sharp map is given by $\sharp_{\omega_L}=\begin{pmatrix}
        0&-\mathcal{G}^{ij}\\
        \mathcal{G}^{ij}&\mathcal{G}^{ij}A^{-1}\mathcal{G}^{ij}
    \end{pmatrix}$ where $\mathcal{G}^{ij}$ is the inverse of $\mathcal{G}_{ij}$ and for any one-form $f\in\Omega^1(Q)$ with $f=f^idq_i$ we have that $\sharp_{\omega_L}f^V=\begin{pmatrix}
        0\\
        \mathcal{G}^{ij}f^j
    \end{pmatrix}$. On the other hand, the sharp map of the Riemannian metric $\mathcal{G}$ is given by $\sharp_\mathcal{G}=\mathcal{G}^{ij}$ and so $[\sharp_{\mathcal{G}}(f)]^V=\sharp_{\omega_L}f^V$. Finally, from 2. of Lemma \ref{completemetricLemma} we have $\sharp_{\omega_L}(f^V)=\sharp_{\mathcal{G}^c}(f^V)$.
\end{proof}

Next, we introduce the main result of the paper.

\begin{theorem}
    A vector field $\Gamma$ of the form \eqref{SODE} corresponding to the closed-loop system of the Lagrangian control system \eqref{mechanical:control:system} makes $\mathcal{M}$ invariant if and only if it satisfies
    \begin{equation}\label{noneq2}
  i_{\Gamma}\omega_{L} - dE_{L} = -\tau^*_{a}(f^a)^V, \,\,\Gamma \in TM,
\end{equation}
or, equivalently, $i_{\Gamma}\omega_{L} - dE_{L} \in \flat_{\mathcal{G}^c}(\mathcal{F}^V)$, where $\flat_{\mathcal{G}^c}(\mathcal{F}^V)=\text{span}\{\flat_{\mathcal{G}^c}(Y^V)\}=\text{span}\{(f^a)^V\}$, and $\mathcal{F}^V$ the distribution on $TQ$ spanned by the vector fields $\{\sharp_{\mathcal{G}^{c}}(f^a)^{V}\}$, being $\tau^*_a$ the unique control law from Theorem \ref{main:theorem}.

\end{theorem}

\begin{proof}

Let $G$ be the vector field defined by the free system $i_{\Gamma_{}}\omega_{L} = dE_{L}$, i.e. $\sharp_{\omega_{L}}(dE_{L})=G$ and $\Gamma$ the one defined by the equation $i_{\Gamma_{}}\omega_{L} - dE_{L} = -u_{a}(f^a)^V$. Thus, $\Gamma$ is of the form 
$\Gamma(v_q)=G(v_q)+u_a(Y^a)^V_{v_q}$, for $v_q\in TQ,$ where $(Y^{a})^{V}=(\sharp_{\mathcal{G}}(f^{a}))^{V}=\sharp_{\mathcal{G}^c}((f^a)^V)=\sharp_{\omega_{L}}((f^{a})^{V})$ and the last equality holds by Lemma \ref{lemma:similar:sharps}.
From Theorem \ref{main:theorem} there exists a unique control function $\tau^*_a$ that makes $\mathcal{M}$ a virtual nonholonomic constraint, i.e. the vector field $\Gamma\in \mathfrak{X}(\mathcal{M})$ satisfies $i_{\Gamma}\omega_{L} - dE_{L} = -\tau^*_{a}(f^a)^V,$ and 
and it is of the form 
\[\Gamma(v_q)=G(v_q)+\tau^*_a(Y^a)^V_{v_q}\in T_{v_q}\mathcal{M},\] for $v_q\in \mathcal{M}.$ Equivalently, $i_{\Gamma}\omega_{L} - dE_{L} \in \flat_{\mathcal{G}^c}(\mathcal{F}^V)$ since for all $a=1,\dots,m$ $\flat_{\mathcal{G}^c}[(Y^a)^V]=\flat_{\mathcal{G}^c}[(\sharp_\mathcal{G}f^a)^V] = \flat_{\mathcal{G}^c}[\sharp_{\mathcal{G}^c}(f^a)^V]  = (f^a)^V$ where we have used property 2. of Lemma \ref{completemetricLemma}. \hfill$\square$
\end{proof}


The next proposition shows that if the vertical lift of the input distribution is orthogonal to the distribution $S$ defined by the nonholonomic system, then the constrained dynamics is precisely the nonholonomic dynamics with respect to the original Lagrangian function.

\begin{proposition}\label{orthogonal:input:distribution}
If $\mathcal{F}^{V}=S^{\bot}$  then the trajectories of the feedback controlled mechanical system \eqref{noneq2} are the nonholonomic equations of motion \eqref{Chetaev's eqns}.
\end{proposition}

\begin{proof}
From the symplectic formulation of Chetaev's equations in (\ref{noneq}) and the definition of the distribution $S$, we have that if $S^\perp$ equals $\mathcal{F}^{V}$ then equation \eqref{noneq2} becomes equation \eqref{noneq}. \hfill$\square$
\end{proof}

\begin{remark}
    Notice that given a mechanical system with nonlinear constraints, there always exist a distribution $\mathcal{F}$ such that $\mathcal{F}^{V}=S^{\bot}$, since $S^{\bot}$ is spanned by vertical lifts of vector fields on $Q$.\hfill$\diamond$
\end{remark}

\begin{remark}
    When $\mathcal{M}$ is a linear distribution on $Q$, the assumption made in the previous proposition reduces to the assumption considered in \cite{virtual}, i.e., $\mathcal{F}$ is orthogonal to $\mathcal{M}$.
    \hfill$\diamond$
\end{remark}

Note that equation \eqref{noneq2} can be equivalently written in the form 
\begin{equation}\label{noneq3}
  i_{\Gamma}\omega_{L} - dE_{L} \in J^{*}\hat{\mathcal{F}}^{o}, \,\,\Gamma \in T\mathcal{M},
\end{equation}
where $\hat{\mathcal{F}}^{o}=\text{span}\{d\hat{f}^{a}\}$ and $\hat{f}^{a}$ are the fiberwise linear functions on $TQ$ defined by $\hat{f}^{a}(v_{q})=\langle f^{a}(q), v \rangle$.

Equations \eqref{noneq3} resemble the symplectic equations appearing in \cite{deLeon} in the context of constrained mechanical systems. Although it is slightly different, many of the constructions obtained by these authors follow in our case. In particular, we can characterize the closed-loop dynamics as the projection of the uncontrolled dynamics to the tangent bundle $T\mathcal{M}$.

Consider the distribution $\mathcal{S}=\sharp_{\omega_{L}}(J^{*}\hat{\mathcal{F}}^{o})$. It is not difficult to prove that $\mathcal{S}=\mathcal{F}^{V}$. In particular, the transversality assumption appearing in Theorem \ref{main:theorem} is equivalent to $\mathcal{S}\cap T\mathcal{M} = \{0\}$, which implies the Whitney sum decomposition
$TTQ|_{\mathcal{M}}=\mathcal{S}|_{\mathcal{M}}\oplus T\mathcal{M}$.

Now, choosing the vector field $\{(Y^{a})^{V}\}$ as a local basis for the distribution $\mathcal{S}$, we can define the associated projections $\mathcal{Q}:T TQ \to \mathcal{S}$ and $\mathcal{P}: T TQ \to T\mathcal{M}$ given by
$\mathcal{Q} = C_{ab}(Y^{a})^V\otimes d\phi^{b}$ and $\mathcal{P}=Id - \mathcal{Q}$, where the matrix $C_{ab}$ is the inverse matrix of
$C^{ab}=(Y^{b})^{V}(\phi^{a})=\mathcal{G}^{ij}f_{i}^{b}\frac{\partial \phi^{a}}{\partial \dot{q}^{j}}=-\sharp_{\omega_{L}}(J^{*}d\hat{f}^{b})(\phi^{a})$. 
Note that this matrix is invertible to the fact that the Riemannian metric is invertible.

Finally we can prove the following result.
\begin{proposition}
    The vector field $\Gamma$ defined by  \eqref{noneq3} satisfies
    $$\Gamma=\mathcal{P}(G)=G+C_{ab}G(\phi^{b})(Y^{a})^{V}.$$
\end{proposition}

\begin{proof}
    The vector field $\Gamma$ satisfies $\Gamma=G+\lambda_{a}(Y^{a})^{V}$. Applying the projection $\mathcal{P}$ to both sides of this equality and using $\Gamma\in T\mathcal{M}$ to impose $\mathcal{P}(\Gamma)=\Gamma$, we deduce that $\tau^*_{a}=C_{ab}G(\phi^{b})$ which proves the result.\hfill$\square$
\end{proof}

Therefore, we have shown that the closed-loop dynamics results from the projection to $T\mathcal{M}$ of the uncontrolled dynamics. 


\begin{remark}
    Note that the symplectic point of view is useful in the study of virtual nonlinear nonholonomic constraints because it provides a geometric framework that captures the intrinsic structure of the system’s dynamics. The symplectic formalism allows for a natural characterization of the constraints in terms of the symplectic structure on $TQ$ induced by the Lagrangian function and the almost-tangent structure. This perspective ensures that the closed-loop vector field satisfies a Chetaev-type geometric equation, linking the constraints directly to the system’s variational properties.
\end{remark}
    
\section{Example: The double pendulum}\label{application section}

Consider the controlled double pendulum. This resembles the well-known acrobot,  but the actuator is at the shoulder rather than at the elbow.The configuration manifold of the system is $Q=\mathbb{S}^1\times\mathbb{S}^1$ with $q=(q_1,q_2)\in Q,$ where $q_1$ represents the angle for the shoulder and $q_2$ the angle for the elbow. 

The Lagrangian $L:TQ\to\R,$ is given by $L(q,\dot{q})=\frac{1}{2}\dot{q}^TD(q)\dot{q} - V(q)$, where $D(q)=\begin{bmatrix}
    ml^2(3+2\cos q_2) & ml^2(1+\cos q_2) \\
    ml^2(1+\cos q_2) & ml^2
\end{bmatrix}$ and $V(q)=-mgl\left(2\cos{q_1}+\cos{(q_1+q_2)}\right)$
are the inertia matrix and the potential, respectively. For simplicity, we set $m=l=1$.
The constraint is given by the equation $\Phi(q,\dot{q})=q_2 - \arctan\Big[(3+2\cos q_2)\dot{q_1} + (1+\cos q_2)\dot{q_2}\Big]$
and the control force is $F(q,\dot{q},u)=udq_1$. 

The controlled Euler-Lagrange equations are 
$D(q)\ddot{q} + P(q,\dot{q})=B$, with \[P(q,\dot{q})=\begin{bmatrix}
    -2s_2\dot{q_1}\dot{q_2} - s_2(\dot{q_2})^2 + g(2s_1 + s_{12}) \\
    -s_2\dot{q_1}\dot{q_2} +gs_{12}
\end{bmatrix}, B=\begin{bmatrix}
    u \\ 0
\end{bmatrix}\]
where, as shorthand, we write $s_1=\sin q_1, s_2=\sin q_2$ and $s_{12}=\sin (q_1 + q_2).$
The constraint manifold is $\mathcal{M}=\{(q,\dot{q})\in TQ \; :\; \Phi(q,\dot{q})=0\}$ and its tangent space is given by $T_{(q,\dot{q})}\mathcal{M}=\{v\in TTQ\; :\; d\Phi(v)=0\} =\spn\{X_1, X_2, X_3\}$, with
$X_1=\frac{\partial}{\partial q_1},  \quad X_2=(1+\cos q_2)\frac{\partial}{\partial \dot{q_1}} - (3+2\cos q_2)\frac{\partial}{\partial\dot{q_2}}$, $X_3=(3+2\cos q_2)\frac{\partial}{\partial q_2}+\Big[A + \sin q_2 (2\dot{q_1}+\dot{q_2})\Big]\frac{\partial}{\partial\dot{q_1}}$, where $A=1+\Big[(3+2\cos q_2)\dot{q_1} + (1+\cos q_2)\dot{q_2}\Big]^2.$


    The input distribution $\mathcal{F}$ is generated by the vector field
    $Y=\frac{\partial}{\partial q_1} - (1+\cos q_2)\frac{\partial}{\partial q_2}$. Note here that the vertical lift of the input distribution, $\mathcal{F}^V$, which is generated by $Y^V= \frac{\partial}{\partial \dot{q_1}} - (1+\cos q_2)\frac{\partial}{\partial \dot{q_2}}$, is transversal to the tangent space of the constraint manifold, $T\mathcal{M}$, thus, by Theorem (\ref{main:theorem}) there is a unique control law making the constraint manifold a virtual nonholonomic constraint.

    The control law that makes the constraint manifold invariant is 
    \[\hat{u}=-\frac{\dot{q}_2(c^2_2-2)(t^2_2-1)+qC}{d}\]
    where $C = s_{12}c_2(2-c^2_2) + s_1c_2(2c_2-1) - 5s_1 - 2s_2c_1$, $d$ is the determinant of the inertia matrix $D(q)$ and $c_1=\cos q_1, c_2=\cos q_2$ and $t_2=\tan q_2$.

We have run a simulation of the double pendulum using a fourth-order Runge-Kutta method for $N=100$ steps using a time step of $h=0.1$ and physical constants $m=l=1$ and $g=10$. We used as initial conditions $q_{1}=0.4, q_{2}=0$ and $\dot{q}_{2}=10$. The velocity $\dot{q}_{2}$ is computed by solving the equation $\Phi(q,\dot{q})=0$, so that initial conditions are in the constraint submanifold. The plot (\ref{angles evolution}) shows the time evolution of the angles $q_{1}$ and $q_{2}$, while the plots (\ref{phase-space q1}) and (\ref{phase-space q2}) show the phase space $(q_{1},\dot{q}_{1})$ and $(q_{2},\dot{q}_{2})$, respectively. The last two figures show the energy and constraint evolution in time. {The simulation shows that the controlled motion has an equilibrium point, occurring near $q_{1}=q_{2}=0$. The system dissipates energy as a result of the control forces acting on it. The constraint is approximately preserved, though we cannot observe exact preservation since the method is not specifically designed to preserve it. It is interesting to note that as we increase the value of the initial angles, the system eventually reaches an equilibrium point but it occurs at points distant from $q_{1}=q_{2}=0$, where a constant control force must always be active.} The reader can see the code and a video of the simulation on the page \url{https://github.com/alexanahory/VNNC}.

    \begin{figure}[htb!]
        \centering
            \begin{subfigure}[b]{0.32\textwidth}
         \centering 
         \includegraphics[width=\textwidth]{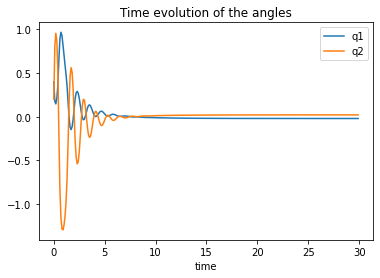}
         \caption{angles time evolution}
         \label{angles evolution}
     \end{subfigure}
     \begin{subfigure}[b]{0.32\textwidth}
         \centering
         \includegraphics[width=\textwidth]{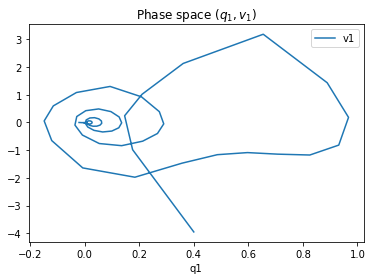}
         \caption{phase space $(q_1,\dot{q}_1)$}
         \label{phase-space q1}
     \end{subfigure}
        \label{fig:my_label}
     \begin{subfigure}[b]{0.32\textwidth}
         \centering
         \includegraphics[width=\textwidth]{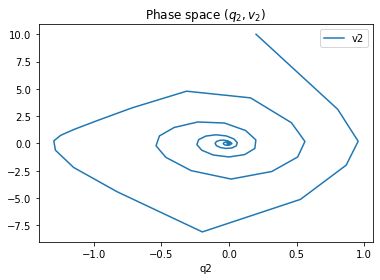}
         \caption{phase space $(q_2,\dot{q}_2)$}
         \label{phase-space q2}
     \end{subfigure}
            \begin{subfigure}[b]{0.32\textwidth}
         \centering
         \includegraphics[width=\textwidth]{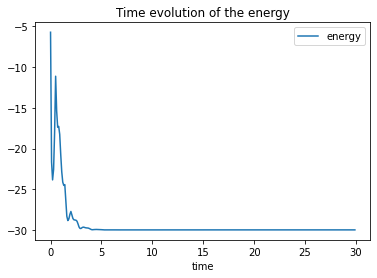}
         \caption{energy in time}
         \label{energy evolv}
     \end{subfigure}
        \label{fig:my_label}
         \centering
        \begin{subfigure}[b]{0.32\textwidth}
         \centering
         \includegraphics[width=\textwidth]{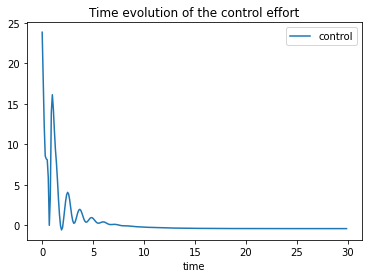}
         \caption{control in time}
         \label{control evolv}
     \end{subfigure}
     \begin{subfigure}[b]{0.32\textwidth}
         \centering
         \includegraphics[width=\textwidth]{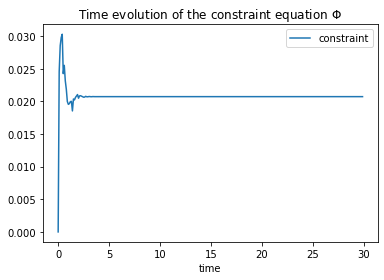}
         \caption{constraint in time}
         \label{constraint evolv}
     \end{subfigure}
        \label{fig:my_label}
    \end{figure}

\end{document}